\documentclass[12pt,fleqn]{article}
\usepackage{amsmath,amsthm,amssymb,amscd}
\usepackage{graphicx}
\usepackage{graphics}
\usepackage{verbatim}

\usepackage{mathrsfs}
\usepackage{color}
\makeatletter

\newcommand{\Rmnum}[1]{\expandafter\@slowromancap\romannumeral #1@}

\newtheorem{theorem}{Theorem}[section]
\newtheorem{proposition}[theorem]{Proposition}

\newtheorem{lemma}[theorem]{Lemma}

\newtheorem{corollary}[theorem]{Corollary}

\newtheorem{conjecture}[theorem]{Conjecture}
\makeatother

\begin{document}

\title{Petersen cores and the oddness of cubic graphs}

\vspace{3cm}
\author{Ligang Jin\footnotemark[1] \footnotemark[2], Eckhard Steffen\footnotemark[1]}
\footnotetext[1]{Paderborn Institute for Advanced Studies in
		Computer Science and Engineering and institute for Mathematics,
		Paderborn University,
		Warburger Str. 100,
		33102 Paderborn,
		Germany; ligang@mail.upb.de (Ligang Jin), es@upb.de (Eckhard Steffen).}
\footnotetext[2]{Supported by Deutsche Forschungsgemeinschaft (DFG) grant STE 792/2-1.}
	
\date{}

\maketitle

\begin{abstract}
{\small Let $G$ be a bridgeless cubic graph. Consider a list of $k$ 1-factors of $G$. Let $E_i$ be the
set of edges contained in precisely $i$ members of the $k$ 1-factors. Let $\mu_k(G)$ be the smallest $|E_0|$ over
all lists of $k$ 1-factors of $G$. We study lists by three 1-factors, and call $G[E_0\cup E_2\cup E_3]$ with $|E_0|=\mu_3(G)$ a $\mu_3(G)$-core of $G$.
If $G$ is not 3-edge-colorable, then $\mu_3(G) \geq 3$. In \cite{Steffen_2014} it is shown that
if $\mu_3(G) \not = 0$, then $2 \mu_3(G)$ is an upper bound for the girth of $G$.
We show that $\mu_3(G)$ bounds the oddness $\omega(G)$ of $G$ as well. We prove that $\omega(G)\leq \frac{2}{3}\mu_3(G)$.
If $\omega(G) = \frac{2}{3} \mu_3(G)$, then every $\mu_3(G)$-core has a very specific structure. We call these cores Petersen cores.
We show that for any given oddness there is a cyclically 4-edge-connected cubic graph $G$ with $\omega(G) = \frac{2}{3}\mu_3(G)$.
On the other hand, the difference between $\omega(G)$ and $\frac{2}{3}\mu_3(G)$ can be arbitrarily big. This is true even if we
additionally fix the oddness.
Furthermore, for every integer $k\geq 3$, there exists a bridgeless cubic graph $G$ such that $\mu_3(G)=k$.}
\end{abstract}

\section{Introduction}
Graphs in this paper may contain multiple edges or loops. An edge is a loop if its two ends are the same vertex.
A 1-factor of a graph $G$ is a spanning 1-regular subgraph of $G$.  Hence, a loop cannot be an edge of a 1-factor, and
a bridgeless cubic graph does not contain a loop.
One of the first
theorems in graph theory, Petersen's Theorem \cite{Petersen_1891} from 1891, states that every bridgeless cubic graph has a 1-factor.

Let $G$ be a cubic graph, $k \geq 1$, and $S_k$ be a list of $k$ 1-factors of $G$. By a list we mean a collection with possible repetition.
For $i \in \{0, \dots ,k\}$  let $E_i(S_k)$
be the set of edges which are in precisely $i$ elements of $S_k$.
We define $\mu_k(G) = \min \{ |E_0(S_k)| \colon S_k \mbox{ is a list of $k$ 1-factors of } G\}$. If there is no harm of confusion, then we
write $E_i$ instead of $E_i(S_k)$.

Let $G$ be a cubic graph and $S_3$ be a list of three 1-factors $M_1, M_2, M_3$ of $G$. Let ${\cal M} = E_2 \cup E_3$, ${\cal U} = E_0$ and
$|{\cal U}| = k$. The edges of $E_0$ are also called the uncovered edges.
The  $k$-core of $G$ with respect to $S_3$ (or to $M_1, M_2, M_3$)  is the subgraph $G_c$ of $G$
which is induced by ${\cal M} \cup {\cal U}$; that is, $G_c = G[{\cal M} \cup {\cal U}]$.
If the value of $k$ is irrelevant, then we
say that $G_c$ is a core of $G$. If $M_1 = M_2 = M_3$, then $G_c = G$. A core $G_c$ is  proper if $G_c \not = G$.
If $G_c$ is a cycle, i.e.~the union of pairwise disjoint circuits, then we say that $G_c$ is a  cyclic core. In \cite{Steffen_2014} it is shown that every bridgeless cubic graph
has a proper core and therefore, every $\mu_3(G)$-core is proper.

Bridgeless cubic graphs which are not 3-edge-colorable are also called snarks.
Sometimes snarks with 2- or 3-edge-cut are considered to be trivial, since they easily reduce to smaller ones, see \cite{Cameron_1987, Gardner_76, Goldberg_1981, Isaacs_1975, Nedela_1996, Steffen_1998}.
Therefore, snarks are sometimes required to be cyclically 4-edge-connected and to have girth at least five.
Clearly, if $G$ is a snark, then $\mu_3(G) > 0$.
Cores are introduced in \cite{Steffen_2014}, and they were used to prove partial results on some
hard conjectures which are related to 1-factors of cubic graphs. In particular, the following conjecture of Fan and Raspaud is
true for cubic graphs $G$ with $\mu_3(G) \leq 6$.

\begin{conjecture} [\cite{Fan_Raspaud_94}] \label{Fan_Raspaud}
Every bridgeless cubic graph has a cyclic core.
\end{conjecture}

Clearly, if $\mu_4(G) = 0$, then $G$ has a cyclic core. If $G$ is Petersen graph, then $\mu_4(G) \not= 0$.
There are infinite families of snarks whose edge set cannot be covered by four 1-factors, see \cite{Esperet_Mazzuoccolo_2013, Haegglund_2012}.

One major parameter measuring the complexity of a cubic graph $G$ is its oddness, which is the minimum number of odd circuits in a 2-factor of $G$. It is denoted by $\omega(G)$.
Cubic graphs with big oddness can be considered as more complicated than those with small oddness. Indeed, many hard conjectures are proved for
cubic graphs with small oddness. In particular, Conjecture \ref{Fan_Raspaud} is true for cubic graphs with oddness at most two, see \cite{M_Skoviera_2014}.
This paper focuses on the structure of cores and relates $\mu_2$, $\mu_3$ to each other and to $\omega$. In \cite{Steffen_2014} it is shown that
the girth of a snark $G$ is at most $2 \mu_3(G)$. We show that $\omega(G) \leq \frac{2}{3}\mu_3(G)$. Hence, $\mu_3$ bounds the oddness of a snark as well.
Furthermore, we show that if $\omega(G) = \frac{2}{3}\mu_3(G)$, then every $\mu_3(G)$-core of $G$ is a
very specific cyclic core. We call these cores Petersen cores.
  We show that there are infinite classes of snarks where equality holds.
These results  imply that if $G$ has a non-cyclic $\mu_3(G)$-core,
then $\omega(G) < \frac{2}{3}\mu_3(G)$.
We also show that there are
snarks $G$ where the difference between $\omega(G)$ and $\frac{2}{3}\mu_3(G)$ is big. This is even true if we additionally fix the oddness.

The last section proves that for every integer $k\geq 3$, there exists a bridgeless cubic graph $G$ such that $\mu_3(G)=k$.

Let $\gamma_2(G) = \min \{|M_1 \cap M_2| : M_1 \text{ and } M_2 \text{ are 1-factors of } G\}$.
Then, $\mu_2(G) = \gamma_2(G) + \frac{1}{3}|E(G)|$.
These parameters also measure the complexity of cubic graphs since $\gamma_2(G) = 0$ if and only if $G$ is 3-edge-colorable.
Clearly, $\omega(G) \leq 2 \gamma_2(G)$ (see \cite{Steffen_2013}).
We prove $2\gamma_2(G)\leq \mu_3(G)-1.$
A further parameter that measures the complexity of a cubic graph is the resistance $r(G)$, which is the cardinality of a minimum color class of a
proper 4-edge-coloring of $G$. The following proposition summarizes some characterizations of 3-edge-colorable cubic graphs.

\begin{proposition}
If $G$ is a cubic graph, then the following statements are equivalent: $(1)$ $G$ is 3-edge-colorable. $(2)~\omega(G)=0$. $(3)~r(G) = 0$. $(4)~\gamma_2(G)=0$.
	$(5)~\mu_3(G)=0$.
\end{proposition}

\section[]{Structure of cores and oddness}  \label{sec_1}

The following property of cores will be used very often in this paper.

\begin{lemma} [\cite{Steffen_2014}] \label{k_core}
Let $k$ be a positive integer. If $G_c$ is a $k$-core of a cubic graph $G$, then $k = |E_2|+2|E_3|$.
\end{lemma}

\begin{theorem} \label{weak_bound} Let $G$ be a bridgeless cubic graph. If $G$ is not 3-edge-colorable,
then $\omega(G) \leq 2\gamma_2(G)\leq \mu_3(G)-1$. Furthermore, if $G$ has a cyclic $\mu_3(G)$-core, then
$\gamma_2(G) \leq \frac{1}{3}\mu_3(G)$.
\end{theorem}

\begin{proof} As already mentioned we have $\omega(G) \leq 2 \gamma_2(G)$.

Let $G_c$ be a $\mu_3(G)$-core of $G$.
By the minimality of $\gamma_2(G)$, we have $3\gamma_2(G)\leq |E_2|+3|E_3|$.
Combining this inequality with $\mu_3(G)=|E_2|+2|E_3|$ (Lemma \ref{k_core}) yields $$2\gamma_2(G)\leq \mu_3(G)-\frac{1}{3}|E_2|.\eqno(1)$$

Hence, the first statement is trivial if $\mu_3(G)$ is odd. If $\mu_3(G)$ is even, then it follows from the fact that $|E_2| \not= 0$, since
$G_c$ is a proper core of $G$.

Furthermore, if $G_c$ is cyclic, then (1) implies that $\gamma_2(G) \leq \frac{1}{3} \mu_3(G)$.
\end{proof}

The bound of Theorem \ref{weak_bound} is achieved by every snark $G$ with $\mu_3(G) = 3$. We will see that there are infinitely many snarks with this property.
For snarks $H$ with $\mu_3(H) > 3$ we will prove a better upper bound for the oddness in terms of $\mu_3(G)$.

Let $G$ be a bridgeless cubic graph and $G_c$ a core of $G$ with respect to three 1-factors $M_1, M_2, M_3$.
$G_c$ is called a Petersen core if the following two conditions hold:

(1)~ $G_c$ is cyclic.

(2)~ If $P$ is a path of length 5 in $G_c$, then there exists no pair of edges $e_1,e_2$ of $P$ and two integers $i,j$ such that $e_1,e_2\in M_i \cap M_j$ and $1\leq i< j \leq 3$.

Let $H$ be a graph. The number of components of $H$ is denoted by $|H|_c$, and the number of odd components is denoted by $|H|_{odd}$.
Furthermore, if $H$ is cubic and $M$ is a 1-factor of $H$, then $\overline{M}$ denotes the complementary 2-factor of $M$.

\begin{theorem} \label{odd circuit-mu3}
Let $G$ be a bridgeless cubic graph. If $G_c$ is a $k$-core of $G$ with respect to three 1-factors $M_1, M_2, M_3$,
then $|\overline{M_1}|_{odd}+|\overline{M_2}|_{odd}+|\overline{M_3}|_{odd}\leq 2k.$
Moreover, if $G_c$ is a $k$-core such that the equality holds, then $G_c$ is a Petersen core.
\end{theorem}

\begin{proof}
\begin{figure}[h]
  \centering
  \includegraphics[width=12cm]{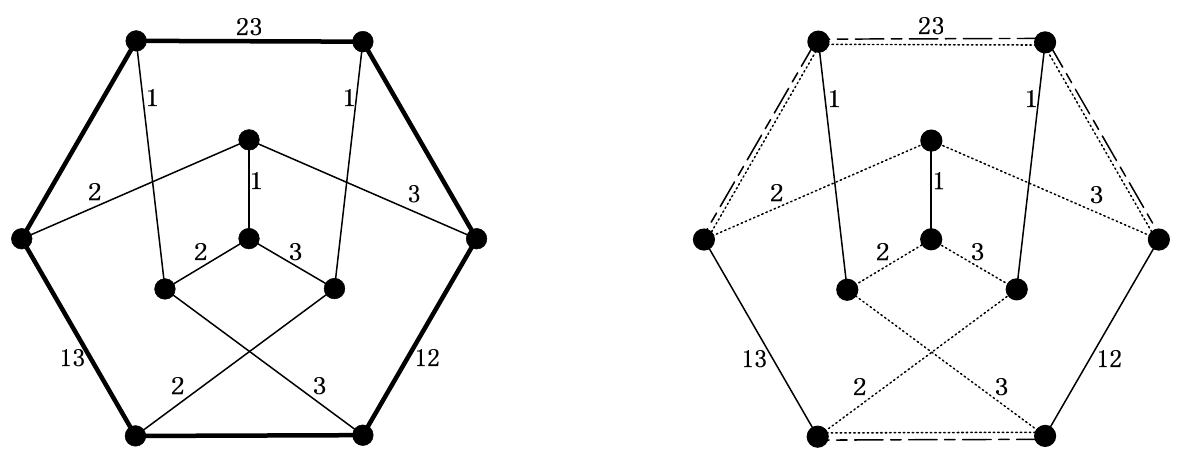}\\
  \caption{the left figure gives a 3-core (in bold line) of Petersen graph where the equality holds, and the right figure gives
  $\overline{M_1}$ (in dotted line) and $\widehat{H_1}$ (in dashed line)}\label{FigThm}
\end{figure}

Let $H$ be a subgraph of $G_c$ which is induced by $E_0\cup E_2$.
Clearly, $H$ consists of pairwise disjoint circuits.
Let $E_{(i)}=E_1\cap M_i$ for $i \in \{1,2,3\}$, and $E_{(i,j)}=E_2\setminus M_l$ for $\{i,j,l\}=\{1,2,3\}$.
We classify the components of $H$ as follows: let $D$ be a component of $H$.
If $D$ contains edges only from $E_0$, then $D$ is class 0. If $D$ is not of class 0 and it contains no edge from $M_i$, then $D$ is class $i$, for $i \in \{1,2,3\}$.
If $D$ is not of class $i$ for all $i \in \{0,1,2,3\}$, then $D$ is class 4.

For $j \in \{0,1,2,3,4\}$ let $Y_j$ be the graph consisting of components of $H$ which are class $j$.

Let $i \in \{1,2,3\}$.
Let $C$ be an odd circuit of $\overline{M_i}$. Then $C$ has at least one uncovered edge, say $e$.
Let $H_i$ be a subgraph of $H$ induced by $E(H)\setminus M_i$.
Clearly, $e\in H_i$. Let $P_e$ be the component of $H_i$ containing $e$.
Since $C$ is a component of $\overline{M_i}$ and since $H_i$ is a subgraph of $\overline{M_i}$, $C$ contains $P_e$.
Furthermore, $P_e$ is either a path or an odd circuit.
Let $\widehat{H_i}$ be the subgraph of $H_i$ consisting of all the components of $H_i$ each of which is either a path or an odd circuit.
It follows that $|\overline{M_i}|_{odd}\leq |\widehat{H_i}|_c$.
Hence, $$\sum_{i=1}^3|\overline{M_i}|_{odd}\leq \sum_{i=1}^3|\widehat{H_i}|_{c}.\eqno(1)$$

Let $D$ be a component of $H$. If $E(D)\cap M_i=\emptyset,$ then $D$ is a component also of $H_i$;
otherwise, the graph induced by $E(D) \setminus M_i$ consists of $|E(D)\cap M_i|$ many disjoint paths and each of these paths is a component of $H_i$.
It follows that $|\widehat{H_i}|_{c}=|Y_0|_{odd}+|Y_i|_{odd}+|E(H)\cap M_i|=|Y_0|_{odd}+|Y_i|_{odd}+|E_2\cap M_i|$.
Hence, $$\sum_{i=1}^3|\widehat{H_i}|_{c}=3|Y_0|_{odd}+\sum_{i=1}^3|Y_i|_{odd}+\sum_{i=1}^3|E_2\cap M_i|.\eqno(2)$$

A vertex $v$ of $G$ is called a bad vertex if $v$ is incident with two uncovered edges.
Clearly, $G$ has precisely $2|E_3|$ many bad vertices.
Since every vertex of $Y_0$ is a bad vertex, $Y_0$ has at least $3|Y_0|_{odd}$ bad vertices.
Let $T$ be any odd component of $Y_1$. Since $T$ is an odd circuit and every edge of $T$ is either uncovered or from $E_{(2,3)}$, it follows that $T$ has at least one pair of adjacent uncovered edges. Hence, $T$ has at least one bad vertex. Thus, $Y_1$ has at least $|Y_1|_{odd}$ bad vertices.
Similarly, for each $j\in \{2,3\}$, $Y_j$ has at least $|Y_j|_{odd}$ bad vertices.
Since $Y_0,Y_1,Y_2,Y_3$ are pairwise disjoint subgraph of $G$, it follows that $Y_0,Y_1,Y_2,Y_3$ have at most $2|E_3|$ bad vertices in total. Thus, $$3|Y_0|_{odd}+\sum_{i=1}^3|Y_i|_{odd}\leq 2|E_3|.\eqno(3)$$

By combining inequalities (1),(2),(3) and the equality $\sum_{i=1}^3|E_2\cap M_i|=2|E_2|$, we conclude that $\sum_{i=1}^3|\overline{M_i}|_{odd}\leq 2|E_2|+2|E_3|$.
By Lemma \ref{k_core} we have $k=|E_2|+2|E_3|$ and therefore, $$\sum\limits_{i=1}^3|\overline{M_i}|_{odd}\leq 2k-2|E_3|\leq 2k.\eqno(4)$$

This completes the first part of the proof.

Now let $G_c$ be a core such that $\sum\limits_{i=1}^3|\overline{M_i}|_{odd}=2k$. By (4), we have $|E_3|=0$. Thus, $G_c$ is a cyclic core.
Furthermore, since $|E_3|=0$, we deduce from (in-)equalities $(1),(2),(3)$ that
$2k=\sum_{i=1}^3|\overline{M_i}|_{odd}\leq \sum_{i=1}^3|\widehat{H_i}|_{c}=2|E_2|$ and from Lemma 2.1 that $k=|E_2|$.
Therefore, $\sum_{i=1}^3|\overline{M_i}|_{odd}= \sum_{i=1}^3|\widehat{H_i}|_{c}$, that is, the inequality (1) becomes an equality.

A path $P$ is bad if it is of odd length and
(a) there is $i \in \{1,2,3\}$ such that $M_i \cap E(P) = \emptyset$, and
(b) the end-vertices of $P$ are incident to an edge of $E_{(i,j)}$, for a $j \in \{1,2,3\} \setminus \{i\}$.

By definition, every bad path of $G$ contains an uncovered edge.

We claim that $G_c$ has no bad path.
Suppose to the contrary that $P$ is a bad path of $G_c$. Without loss of generality, suppose that $E(P)\cap M_1=\emptyset$ and both end-vertices of $P$ are incident with an edge from $E_{(1,2)}$.
Thus $P$ is a component of $\widehat{H_1}$.
Let $C$ be the circuit of $\overline{M_1}$ containing $P$.
Since $\sum_{i=1}^3|\overline{M_i}|_{odd}= \sum_{i=1}^3|\widehat{H_i}|_{c}$, it follows that $C$ is of odd length and contains no other component of $\widehat{H_1}$.
This implies that $C-E(P)$ is a path of even length with edges from $E_{(2)}$ and from $E_{(3)}$ alternately.
But then $P$ has an end-vertex incident with an edge from $E_{(2)}$ and with an edge from $E_{(1,2)}$, a contradiction.


It remains to show that $G_c$ is a Petersen core. Suppose to the contrary that $G_c$ is not a Petersen core. Then $G_c$ violates the second part of the definition of a Petersen core.
Without loss of generality, we may assume that $Q=uvwxyz$ is a path of length 5 in $G_c$ and $e_1, e_2$ are two edges of $Q$ such that $e_1, e_2 \in E_{(1,2)}.$
It suffices to consider the following two cases.

Case 1: $e_1=uv$ and $e_2=wx$. Then $vw$ is a bad path of $G_c$, a contradiction.

Case 2: $e_1=uv$, $e_2=yz$, and $wx \not \in E_{(1,2)}$. Then $vwxy$ is a bad path of $G_c$, a contradiction.

This completes the proof.
\end{proof}

\begin{corollary} \label{omega-mu3}
If $G$ is a bridgeless cubic graph, then $\omega(G)\leq \frac{2}{3}\mu_3(G)$.
Moreover, if $\omega(G) = \frac{2}{3}\mu_3(G)$, then $\omega(G) = 2 \gamma_2(G)$ and every $\mu_3(G)$-core is a Petersen core.
\end{corollary}

\begin{proof}
Let $G_c$ be a $\mu_3(G)$-core of $G$ with respect to three 1-factors $M_1, M_2, M_3$.
By Theorem \ref{odd circuit-mu3}, we have $|\overline{M_1}|_{odd}+|\overline{M_2}|_{odd}+|\overline{M_3}|_{odd}\leq 2\mu_3(G)$.
It follows that $\omega(G)\leq \frac{2}{3}\mu_3(G)$ by the minimality of $\omega(G)$.

If $\omega(G) = \frac{2}{3}\mu_3(G)$, then $|\overline{M_1}|_{odd}+|\overline{M_2}|_{odd}+|\overline{M_3}|_{odd} = 2\mu_3(G)$.
Again by Theorem \ref{odd circuit-mu3}, $G_c$ is a Petersen core.
By Theorem \ref{weak_bound}, $\gamma_2(G) \leq \frac{1}{3} \mu_3(G)$.
Therefore, $\omega(G) \leq 2 \gamma_2(G) \leq \frac{2}{3} \mu_3(G) = \omega(G)$. Hence, $\omega(G) = 2 \gamma_2(G)$.
\end{proof}

\section[]{Snarks with specific cores}  \label{sec_1}

This section shows that there exists an infinite class of cubic graphs $G$ with $\omega(G)=\frac{2}{3}\mu_3(G)$. Hence, the upper bound $\frac{2}{3}\mu_3(G)$ for $\omega(G)$ is best possible.

A network is an ordered pair $(G,U)$ consisting of a graph $G$ and a subset $U \subseteq V(G)$ whose elements are called terminals.
A network with $k$ terminals is a $k$-pole. We consider networks $(G,U)$ with $d_G(v) = 1$ if $v$ is a terminal and $d_G(v) = 3$ otherwise.
A terminal edge is an edge which is incident to a terminal vertex.

For $i \in \{1,2\}$ let $T_i$ be a network and $u_i$ be a terminal of $T_i$. The junction of $T_1$ and $T_2$ on $(u_1,u_2)$ is the network obtained from
$T_1$ and $T_2$ by identifying $u_1$ and $u_2$ and suppressing the resulting bivalent vertex.

\begin{theorem} \label{G_k}
For every positive integer $k$, there is a cyclically 4-edge-connected cubic graph $G_k$ of order $26k$ and $\omega(G_k)=r(G_k)=2\gamma_2(G_k)=\frac{2}{3}\mu_3(G_k)=2k$.
\end{theorem}

\begin{proof} We will construct graphs with these properties.
Let $B$ be a 4-pole with terminals $a,b,c,d$ as shown in Figure \ref{NetworkB}.
Take $k$ copies $B_0,\dots,B_{k-1}$ of $B$.
Let $G_k$ be the junction of $B_0,\dots,B_{k-1}$ on $(c_i,a_{i+1})$ and $(d_i,b_{i+1})$ for $i \in \{0,\dots,k-1\}$, where the indices are added modulo $k$
(Figure \ref{Fig_G2} illustrates $G_2$ and a $\mu_3(G_2)$-core in bold line).
\begin{figure}[hh]
  \centering
  \includegraphics[width=10cm]{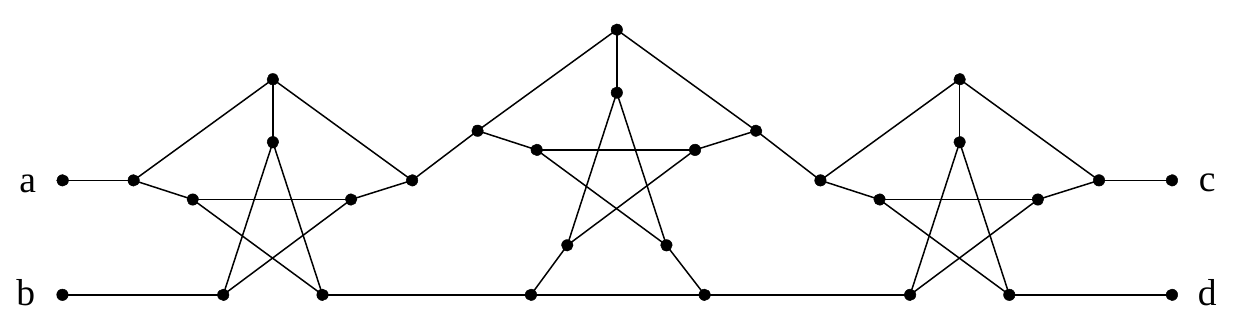}\\
  \caption{4-pole $B$}\label{NetworkB}
\end{figure}

\begin{figure}[hh]
  \centering
  \includegraphics[width=9cm]{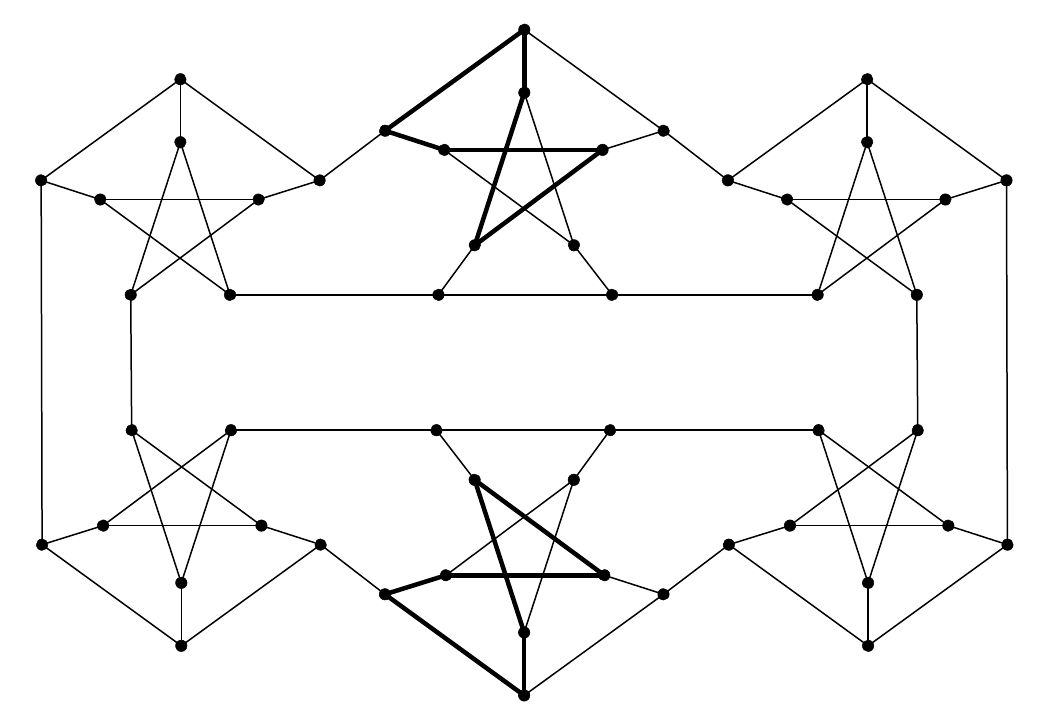}\\
  \caption{$G_2$ and a $\mu_3(G_2)$-core in bold line}\label{Fig_G2}
\end{figure}

It is easy to check that $r(B)=2$. Hence, we have $r(G_k)\geq 2k$.
Furthermore,
let $M'_i,M''_i,M'''_i$ be three matchings of $B_i$ as shown in Figure \ref{Fig_3Coloring_B} labeled with numbers $1, 2, 3$, respectively. Consider these matchings as matchings in
$G_k$, where the edges with the suppressed bivalent vertices belong to $M'_1$.
Let $M'=\bigcup_{i=0}^{k-1} M'_i$, $M''=\bigcup_{i=0}^{k-1} M''_i$, $M'''=\bigcup_{i=0}^{k-1}M'''_i$.
Then $M', M'', M'''$ are three 1-factors of $G_k$, and $G_k$ has precisely $3k$ edges contained in none of $M', M''$ and $M'''$. Hence, we have $\mu_3(G_k)\leq 3k$.
Since $\omega(G_k) \leq \frac{2}{3}\mu_3(G_k)$ by Corollary \ref{omega-mu3}, it follows that $2k\leq r(G_k) \leq \omega(G_k) \leq \frac{2}{3}\mu_3(G_k)\leq 2k$. Therefore, we have $\omega(G_k)=r(G_k)=\frac{2}{3}\mu_3(G_k)=2k = 2\gamma_2(G_k)$, where the last equality follows by Corollary \ref{omega-mu3}.

\begin{figure}[hh]
  \centering
  \includegraphics[width=12cm]{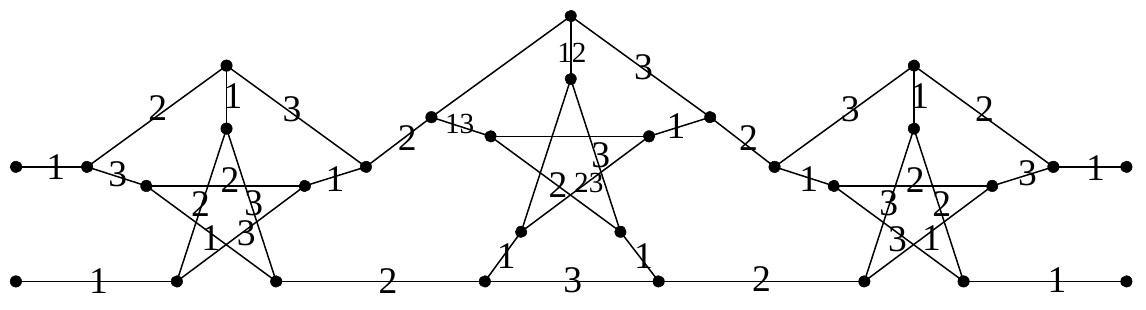}\\
  \caption{three matchings of $B_i$}\label{Fig_3Coloring_B}
\end{figure}
\end{proof}

Lukot'ka et al. \cite{Skoviera_2013} constructed a cyclically 4-edge-connected snark with oddness 4 and forty-four vertices, see Figure 4.
So far it is the smallest known with these properties.
They conjectured that every smallest cyclically 4-edge-connected snark with oddness 4 has forty-four vertices. For the graph of Figure 4
holds $\omega(G)=\frac{2}{3}\mu_3(G)$. The proof, which we omit here, is similar to the one of Theorem \ref{G_k}.

\begin{figure}[hh]
  \centering
  \includegraphics[width=6cm]{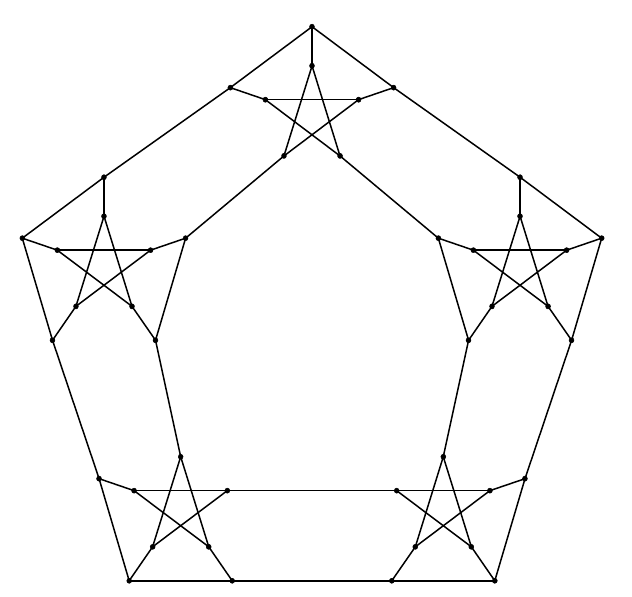}\\
  \caption{A cyclically 4-connected snark of order 44 with oddness 4 \cite{Skoviera_2013}}\label{smallest}
\end{figure}

For $i \in \{1,2\}$ let $H_i$ be two cubic graphs and $u_i$ be a vertex of $H_i$ with neighbors $x_i,y_i,z_i$.
A 3-junction $G$ of $H_1$ and $H_2$ on $(u_1,u_2)$ is the graph obtained from $H_1$ and $H_2$ by deleting vertices $u_1$ and $u_2$, and adding new edges $x_1x_2, y_1y_2$ and $z_1z_2$.
The set $\{x_1x_2, y_1y_2, z_1z_2\}$ is called a 3-junction-cut of $G$ with respect to $H_1$ and $H_2$.

Next we show that the difference between the oddness of a cubic graph $G$ and $\frac{2}{3}\mu_3(G)$ can be arbitrary big. We will use the following theorem
which is a simple consequence of a result of A.~Weiss.

\begin{theorem} [\cite{Weiss_1984}] \label{Thm_HugeGirth}
For every positive integer $c$ there is a connected bipartite cubic graph $H$ with $girth(H) \geq c$.
\end{theorem}

\begin{lemma} [\cite{Steffen_2014}] \label{girth}
Let $G$ be a bridgeless cubic graph. If $G$ is not 3-edge-colorable, then $girth(G) \leq 2 \mu_3(G)$.
\end{lemma}

\begin{theorem}
For any positive integers $k$ and $c$, there exists a bridgeless cubic graph $G$ with $\omega(G) = 2k$ and $\mu_3(G)\geq c$.
\end{theorem}

\begin{proof} By Theorem \ref{G_k} there is a cyclically 4-edge-connected cubic graph $H$ with $\omega(H) = 2k = \frac{2}{3} \mu_3(G)$. Hence,
we are done for $c \leq 3k$.

Let $V(H) = \{v_1,\dots,v_n\}$.
By Theorem \ref{Thm_HugeGirth}, there is a connected bipartite cubic graph $T$ with $girth(T)\geq 2c$.
Since every bipartite cubic graph has no bridge, $T$ is bridgeless.
Take $n$ copies $T_1,\dots,T_n$ of $T$, and let $u_i$ be a vertex of $T_i$. Let $H_0 = H$ and for $i \in \{1, \dots, n\}$
let $H_{i}$ be a 3-junction of $H_{i-1}$ and $T_{i}$ on $(v_i, u_i)$, and let $G=H_n$.

We claim that $\omega(H_i) = \omega(H_{i-1})$.
Let $M$ be a 1-factor of $H_{i-1}$ such that $\overline{M}$ has $\omega(H_{i-1})$ odd circuits.
Precisely one edge of $M$ is incident to $v_i$. Since $T_i$ is bridgeless cubic and bipartite, it follows that $\overline{M}$
can be extended to a 2-factor of $H_i$ that has $\omega(H_{i-1})$ many odd circuits. Hence,
$\omega(H_i) \leq \omega(H_{i-1})$.

Let $F$ be a 1-factor of $H_i$ such that $\overline{F}$ has $\omega(H_i)$ many odd circuits.
Let $J$ be the 3-junction-cut of $H_i$ with respect to $H_{i-1}$ and $T_i$.
If $F$ contains all edges of $J$, then every circuit of $\overline{F}$ lies either in $H_i - v_{i-1}$ or in
$T_i - u_i$. Since the order of $H_i[V(T_i) \setminus \{u_i\}]$ is odd, it follows that $H_i[V(T_i) \setminus \{u_i\}]$ contains a circuit of odd length, contradicting the fact that $T_i$ is bipartite. Hence,
$F$ contains precisely one edge of $J$. Then $F$ can be transformed to a 1-factor of $H_{i-1}$ by contracting $T_i$ to a vertex.
Since the complement of this 1-factor has
at most $\omega(H_i)$ odd circuits, it follows that $\omega(H_{i-1}) \leq \omega(H_i)$. Therefore, $\omega(G)=\omega(H)$.

By construction we have $girth(G) \geq 2c$, and therefore, $\mu_3(G)\geq c$ by Lemma \ref{girth}.
\end{proof}

Goldberg snarks \cite{Goldberg_1981} and Isaacs flower snarks \cite{Isaacs_1975} are two well known families of snarks. The proof of the following proposition is easy.

\begin{proposition}
If $G$ is a flower snark or a Goldberg snark, then $\omega(G)=2$ and $\mu_3(G)=3$.
\end{proposition}

Since highly cyclically edge-connected snarks are of general interest, we prove the following statement.
\begin{theorem} \label{Thm_G^k}
For every positive integer $k$, there is a cyclically 5-edge-connected cubic graph $G^k$ with $\mu_3(G^k) = 2 \omega(G^k) = 4k$.
\end{theorem}

\begin{proof} We will construct graphs with these properties.

Let $D$ be a 5-pole with terminals $u,v,w,x,y$ as shown in Figure \ref{Fig_NetworkD}.
Take $2k$ copies $D_1,\dots,D_{2k}$ of $D$, and let $G^{k}$ be the junction of $D_1,\dots,D_{2k}$ on $(x_i,u_{i+1})$ and $(y_i,v_{i+1})$ for $i \in \{1,\dots,2k\}$ and on $(w_i,w_{i+k})$ for $i \in \{1,\dots,k\}$ (Figure \ref{Fig_G^2} illustrates $G^2$).

\begin{figure}[hh]
  \centering
  \includegraphics[width=10cm]{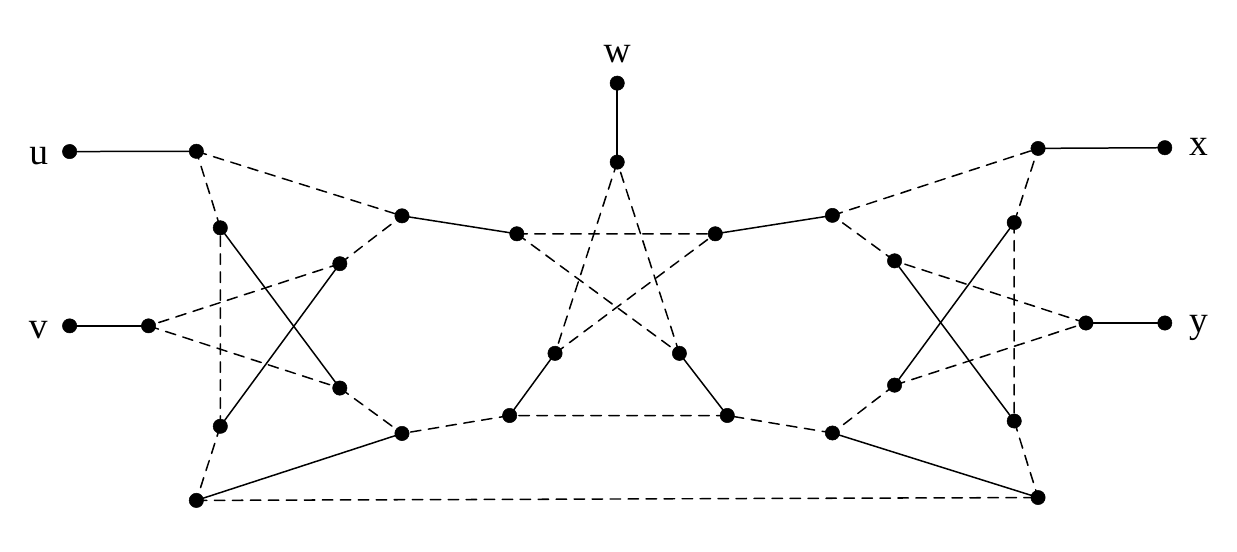}\\
  \caption{The 5-pole $D$ and a 2-regular subgraph $S$ of $D$ in dotted line}\label{Fig_NetworkD}
\end{figure}

\begin{figure}[hh]
  \centering
  \includegraphics[width=12cm]{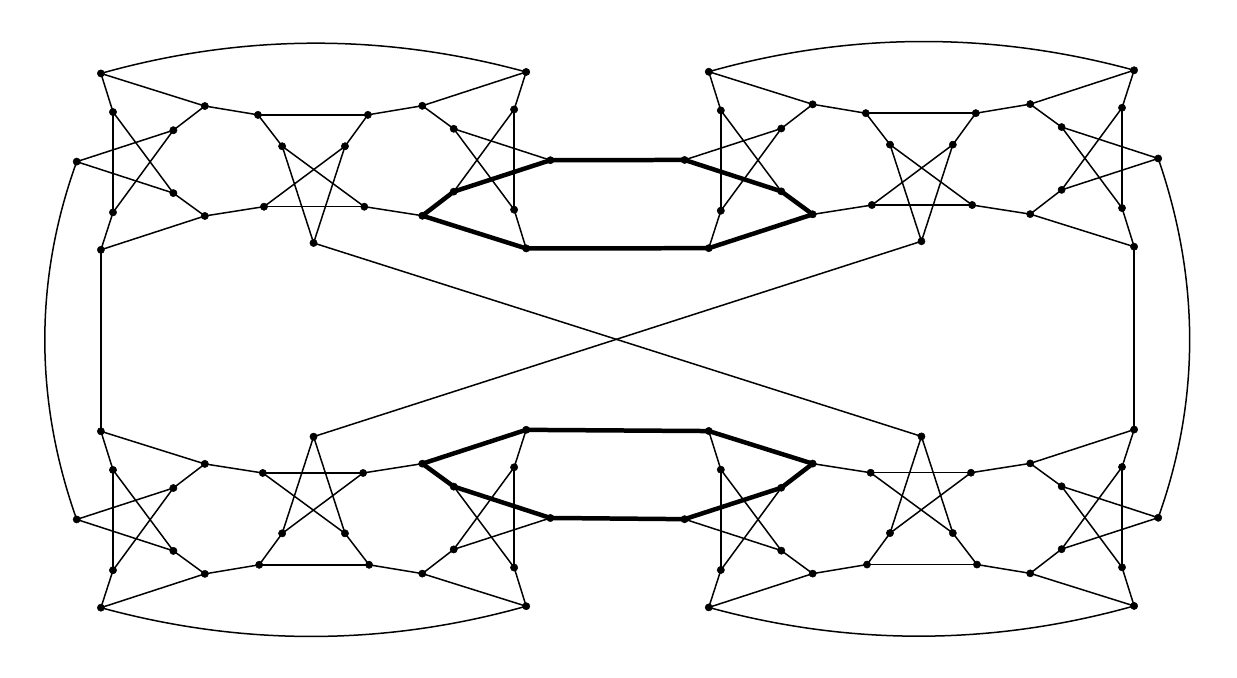}\\
  \caption{$G^2$ and a $\mu_3(G^2)$-core of $G^2$ in bold line}\label{Fig_G^2}
\end{figure}

We claim that $G^{k}$ is a cyclically 5-edge-connected cubic graph such that $\omega(G^k) = 2k$ and $\mu_3(G^k) = 4k$.

Since $D$ is not 3-edge-colorable, every cover by three matchings leaves at least one edge uncovered. Thus, $r(G^k)\geq 2k$ and $\omega(G^k)\geq 2k$.

Let $S_i$ be a set of edges of $D_i$ as shown in Figure \ref{Fig_NetworkD} and let $F=\bigcup_{i=1}^{2k}S_i$.
It is easy to see that $F$ is a 2-factor of $G^k$ that contains precisely $2k$ odd circuits. Thus, $\omega(G^k)\leq 2k$ and therefore, $\omega(G^k) = 2k$.

Let $\widetilde{D_i}$ be the junction of $D_{2i-1}$ and $D_{2i}$ on $(x_{2i-1},u_{2i})$ and $(y_{2i-1},v_{2i})$ ($i \in \{1,2,\dots,k\}$), and $M'_i,M''_i,M'''_i$ be three matchings of $\widetilde{D_i}$ as shown in Figure \ref{Fig_DD_123} labeled with numbers $1, 2, 3$, respectively. Let $M'=\bigcup_{i=1}^kM'_i$, $M''=\bigcup_{i=1}^kM''_i$, $M'''=\bigcup_{i=1}^kM'''_i$.
The three 1-factors $M',M'',M'''$ cover all but $4k$ edges of $G^k$. Hence, $\mu_3(G^k)\leq 4k$.

On the other hand, let $G_c$ be a $\mu_3(G^k)$-core of $G^k$. Since each $D_i$ is not 3-edge-colorable,
it has at least one uncovered edge of $G_c$, say $e_i$. Let C be any circuit of $G_c$ containing precisely $t$ members of $\{e_1,\ldots,e_{2k}\}$.
First suppose that $t=1$. Since the girth of $G^k$ is at least 5, it follows that $|E(C)|\geq 5$. Next suppose that $t\geq 2$.
Clearly, each path of $D_i$ joining any two terminals of $D_i$ is of length at least 3. Since $C$ goes through $t$ members of $\{D_1,\ldots,D_{2k}\}$, $|E(C)|\geq 4t$. In both cases we have $|E(C)|\geq 4t$
and thus, $C$ contains at least $2t$ uncovered edges.
Since each $e_i$ lies on precisely one circuit of $G_c$, it follows that $G_c$ contains at least $4k$ uncovered edges. Thus, $\mu_3(G^k)\geq 4k$, and therefore, $\mu_3(G^k) = 4k$.
\begin{figure}[hh]
  \centering
  \includegraphics[width=17cm]{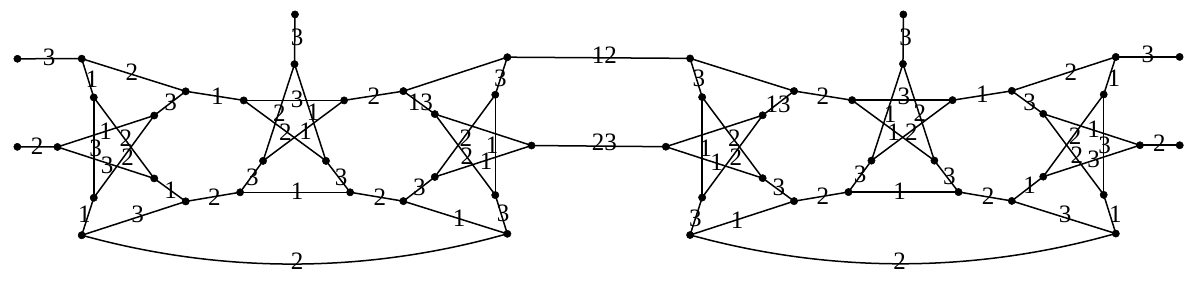}\\
  \caption{The 6-pole $\widetilde{D_i}$
           and three matchings $M'_i,M''_i,M'''_i$ of $\widetilde{D_i}$ labeled with numbers 1,2,3, respectively.}\label{Fig_DD_123}
\end{figure}
\end{proof}

Theorems \ref{G_k} and \ref{Thm_G^k} imply that for every positive integer $k$ with $k\equiv 0 ~(mod ~3)$
there exists a cyclically 4-edge-connected cubic graph $G$ with $\mu_3(G)=k$, and for every positive integer $k$ with $k\equiv 0 ~(mod ~4)$
there exists a cyclically 5-edge-connected cubic graph $G$ with $\mu_3(G)=k$.
We will prove that for every $k \geq 3$ there is a bridgeless cubic graph $H$ with $\mu_3(H) = k$.

Let $G'$ and $G''$ be two bridgeless cubic graphs that are not 3-edge-colorable.
Let $e'=xy$ be an uncovered edge of a $\mu_3(G')$-core of $G'$,
and $e''=uv$ be an uncovered edge of a $\mu_3(G'')$-core of $G''$.
A 2-junction of $G'$ and $G''$ is the graph $G$ with $V(G)=V(G')\cup V(G'')$ and $E(G)=E(G')\cup E(G'')\cup \{ux,vy\}\setminus \{e',e''\}.$
The set $\{ux,vy\}$ is called the 2-junction-cut of $G$ with respect to $G'$ and $G''$.

\begin{lemma} \label{Thm_2junction}
Let $G'$ and $G''$ be two bridgeless cubic graphs that are not 3-edge-colorable. If $G$ is a 2-junction of $G'$ and $G''$, then $\mu_3(G)=\mu_3(G')+\mu_3(G'')$.
\end{lemma}

\begin{proof} By construction, $G$ has a $k$-core with $k \leq \mu_3(G')+\mu_3(G'')$.
Hence, $\mu_3(G)\leq \mu_3(G')+\mu_3(G'')$.

Suppose to the contrary that $\mu_3(G) < \mu_3(G')+\mu_3(G'')$.
Let $ux, vy$ be the 2-junction-cut of $G$ with respect to $G'$ and $G''$, and $u,v\in V(G')$ and $x,y\in V(G'')$.
Let $G_c$ be a $\mu_3(G)$-core of $G$ with respect to three 1-factors $M_1, M_2, M_3$. Then each $M_i$ contains either none of $ux$ and $vy$ or both of them.
Furthermore, $M_i$ induces 1-factors $F_i'$ and $F_i''$ in $G'$ and $G''$, respectively. It follows that there is a $k$-core either in $G'$ with $k < \mu_3(G')$ or in $G''$ with $k < \mu_3(G'')$, a contradiction.
\end{proof}

\begin{theorem}
For every integer $k\geq 3$, there exists a bridgeless cubic graph $G$ such that $\mu_3(G)=k$.
\end{theorem}

\begin{proof}
Let us first consider the case $k \neq 5$. Then there exist two non-negative integers $k'$ and $k''$ such that $k=3k'+4k''$.
By theorems \ref{G_k} and \ref{Thm_G^k}, there is a cyclically 4-edge-connected cubic graph $H'$ with $\mu_3(H')=3k'$ and a
cyclically 5-edge-connected cubic graph $H''$ with $\mu_3(H'')=4k$.
If $k'=0,$ then take $G=H''$ as desired. If $k''=0$, then take $G=H'$ as desired.
Hence, we may next assume that $k',k''>0$.
Let $G$ be a 2-junction of $H'$ and $H''$. By Lemma \ref{Thm_2junction}, $\mu_3(G)=\mu_3(H')+\mu_3(H'')=k$, we are done.

It remains to consider the case $k=5$. Consider the flower snark $J_7$, see Figure \ref{J7}.
Let $J_c$ be a $\mu_3(J_7)$-core of $J_7$. Note that $J_c$ is a circuit of length 6. Let  $u$ be a vertex of $J_c$ and $v,w,x$ be its three neighbors in $J_7$.
Take two copies $J',J''$ of $J_7$. Let $G$ be a 3-junction of $J'$ and $J''$ on $\{u,u''\}$ where $\{v'v'',w'w'',x'x''\}$ is the 3-junction-cut.
This operation yields a core $G_c$ of $G$ and $G_c$ is a circuit of length 10. Hence, $\mu_3(G)\leq 5$.

On the other hand, let $T$ be any $\mu_3(G)$-core of $G$.
By the structure of $T$ as a core, if $\{v'v'', w'w'', x'x''\}\cap (E_0\cup E_2)=\emptyset$,
then both $J'$ and $J''$ contain a circuit of $T$.
Since the girth of $J_7$ is 6, it follows that $T$ has at least six uncovered edges, a contradiction.
Hence, we may assume that $v'v''\in E_0\cup E_2$. Let $C$ be the circuit of $T$ containing $v'v''$.
Clearly, $C$ goes through both $J'$ and $J''$. Since again the girth of $J$ is 6, $C$ is of length at least 10. It follows that $T$ has at least five uncovered edges and thus, $\mu_3(G)\geq 5$. Therefore, $\mu_3(G) = 5$ and every $\mu_3(G)$-core of $G$ is a circuit of length 10.
\end{proof}

\begin{figure}[hh]
  \centering
  \includegraphics[width=4.5cm] {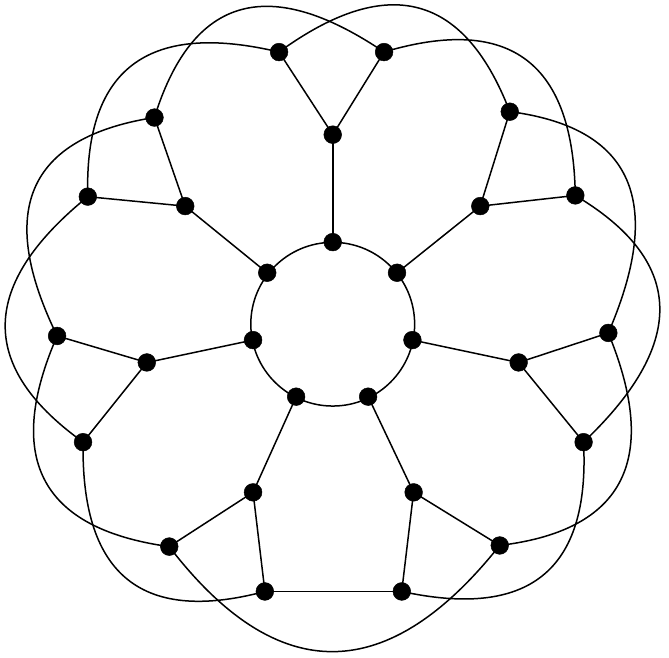}\\
  \caption{flower snark $J_7$}\label{J7}
\end{figure}

\end{document}